\documentclass[11pt]{article}
\makeatletter
\def\@seccntformat#1{\csname the#1\endcsname.\hspace*{0.5em}}

\def\@maketitle{
  \newpage
  \null
  \vskip 2em%
  \begin{center}%
  \let \footnote \thanks
    {\LARGE \@title \par}%
    \vskip 1.5em%
    {\large
      \lineskip .5em%
      \begin{tabular}[t]{c}%
        \@author
      \end{tabular}\par}%
  \end{center}%
  \par
  \vskip 1.5em}
\makeatother

\usepackage{amsmath}
\usepackage{amsfonts}
\usepackage{amssymb}

\usepackage{url}
\usepackage{xspace}
\usepackage{amsbsy}%
\usepackage{algmacro} 
\usepackage{epsfig}%

\usepackage{theorem}
\begingroup \makeatletter
\gdef\th@plain{\normalfont\itshape
  \def\@begintheorem##1##2{%
        \item[\hskip\labelsep \theorem@headerfont ##1\ ##2.]}%
\def\@opargbegintheorem##1##2##3{%
   \item[\hskip\labelsep \theorem@headerfont ##1\ ##2\ (##3).]}}
\endgroup

\def\citet{\citep} 

\newtheorem{theorem}{Theorem}

\newtheorem{lemma}{Lemma}

\newtheorem{defination}{Definition}
\newtheorem{prob}{Problem}

\newcommand{\mon}{{m_d(x)}}
\newcommand{\bX}{\widetilde{X}}
\newcommand{\bL}{\widetilde{\Lambda}}
\newcommand{\bl}{\widetilde{\lambda}}

\newcommand{\diag}{{\text{\upshape diag}}}

\newcommand{\rank}{{\text{\upshape rank}}}
\newcommand{\trace}{{\text{\upshape Tr}}}

\newcommand{\Tr}{{\text{\upshape Tr}}}
\newcommand{\opt}{{\text{\upshape opt}}}
\newcommand{\err}{{\text{\upshape error}}}

\newcommand{\RR}{{\mathbb{R}}}
\newcommand{\QQ}{{\mathbb{Q}}}

\newcommand{\A}{{\mathcal {A}}}
\newcommand{\D}{{{\mathcal {T}}}}
\newcommand{\SSn}{\mathbb{S}_+^n}
\newcommand{\Sn}{\mathbb{S}^n}

\newcommand{\tildeU}{\widetilde{U}}

\newenvironment{proof}{\textit{Proof. }}{\hfill $\square$}
\newenvironment{notation}{\textbf{Notations: }}

\title{The Minimum-Rank Gram Matrix Completion via Modified Fixed Point Continuation Method}

\author{
Yue Ma and Lihong Zhi
\\
Key Laboratory of Mathematics Mechanization \\[-0.5ex]
  Academy of Mathematics and System Sciences \\[-0.5ex]
  Academia Sinica, Beijing 100190, China
\\[-0.5ex]
 {\ttfamily  \{yma,lzhi\}@mmrc.iss.ac.cn}}

\begin{document}
\maketitle

\begin{abstract}
\noindent The problem of computing a representation for a real
polynomial as a sum of minimum number of squares of polynomials  can
be casted as finding  a symmetric positive semidefinite real  matrix
(Gram matrix) of minimum rank subject to linear equality
constraints.
 In this paper, we
propose  algorithms for solving the minimum-rank Gram matrix
completion problem, and show the convergence of these algorithms.
Our methods are based on the modified fixed point continuation (FPC)
method.  We  also use the  Barzilai-Borwein (BB) technique  and a
specific linear combination of two previous iterates to accelerate
the convergence of modified FPC algorithms. We demonstrate the
effectiveness of our algorithms for computing approximate and exact
rational sum of squares (SOS) decompositions of polynomials with
rational coefficients.
\end{abstract}


\section{Introduction}

Let  $x=[x_1, \ldots, x_s]$ and $f\in \RR[x]$, then $f$ is a
 sum of squares (SOS) in $\RR[x]$ if and only if  it can be written in the
form
\begin{equation}\label{sos}
f(x)=\mon^T \cdot W  \cdot \mon,
\end{equation}
in which $\mon$ is a column vector of monomials of degree less than
or equal to $d$ and $W$ is a real positive semidefinite matrix
\cite[Theorem 1]{PW98} (see also \cite{CLR95}). $W$ is also called a
{\it Gram matrix} for $f$. If $W$ has rational entries, then $f$ is
a sum of squares in $\QQ[x_1, \ldots, x_n]$.

\begin{prob}\label{problem1}
Let $f\in \QQ[x_1, \ldots, x_s]$ be a polynomial of the  degree $2
d$, compute a representation for it as a sum of minimum number of
squares of polynomials in $\QQ[x_1,\ldots,x_s]$.
\end{prob}

 The set of all matrices $W$ for which
(\ref{sos}) holds is an affine subspace of the set of symmetric
matrices. If the intersection of this affine subspace with the cone
of positive semidefinite (PSD) matrices is nonempty, then $f$ can be
written as a sum of squares. Since the components of $\mon$ are not
algebraically independent, $W$ is in general not unique. Problem
\ref{problem1} can be restated as finding a  Gram matrix with
minimum rank satisfying a given set of constraints:
\begin{align}\label{min_rank1}
\left.\begin{array}{cl}
  \min & ~\rank(W)   \\
  s.t. & ~f(x) = \mon^T \cdot W \cdot \mon\\
       & ~W\succeq 0, \, W^T = W
 \end{array}
\right\}
\end{align}

For $s=1$, Pourchet's main theorem \cite{Pourchet71}  implies  that
every positive definite univariate polynomial in $\QQ[x]$ is a sum
of five squares in $\QQ[x]$. Therefore, the  minimum rank of the
Gram matrix satisfying  (\ref{min_rank1}) is bounded by $5$ for
$s=1$. For $s>1$, Pfister's general theorem \cite{Pfister67} shows
that every positive definite polynomial in $\RR[x_1, \ldots, x_s]$
is a sum of $2^s$  squares of rational functions in $\RR(x_1,
\ldots, x_s)$. It is well known that there exist positive
semidefinite polynomials which cannot be written as sums of
polynomial squares. However, as shown in \cite{KLYZ09}, various
exceptional SOS problems in the literature by Motzkin, Delzell,
Reznick, Leep and Starr, the IMO'71 problem by A. Lax and P. Lax,
and the polynomial Vor2 in \cite{ELLS07}  can be written as sums of
less than 10 squares of polynomials after multiplying by suitable
polynomials. The advantage of computing a numerical Gram matrix with
small rank is that we can refine the approximately computed Gram
matrix to high accuracy  by structure preserved Gauss-Newton
iteration more efficiently in order to  recover the exact SOS
representation of $f$ \cite{ KLYZ08, KLYZ09,PePa07, PePa08}.

In general, the rank minimization is an intractable problem and is
in fact provably NP-hard due to the combinational nature of the
non-convex rank function \cite{CG84}.  In \cite{Fazel02,
FHB01,RFP08}, they showed that $\rank(W)$ can be replaced by the
nuclear norm of $W$, which is the best convex approximation of the
rank function over the unit ball of matrices.
 Expanding the right-hand side of the equality condition of
(\ref{min_rank1}), matching coefficients of the monomials, we obtain
a set of linear equations for the entries of $W$ which can be
written as ${\A}(W) = b$, where the action of the linear operator
$\A : \Sn \rightarrow\RR^p$ on $W$ is described by $\trace(A_i^{T}
W), i=1,\ldots,p$ for $A_1,\ldots,A_p\in\RR^{n \times n}$. We use
$\A^*: \RR^p \rightarrow \Sn$ to denote the adjoint operator of
$\A$. 
The rank minimization problem (\ref{min_rank1}) can be relaxed to
the nuclear norm minimization problem
\begin{align}\label{min_nuclear1}
\left.\begin{array}{cl}
            \min & ~\|W\|_* \\
            s.t. & ~\A(W) = b\\
                 & ~W\succeq0, \, W^T = W
\end{array}
\right\}
\end{align}
where the nuclear norm $\|W\|_* $ is defined as the sum of its
singular values. The constraint $\A(W)=b$ can also  be relaxed,
resulting in either the problem
\begin{align}\label{min_nuclear2}
\left.\begin{array}{cl}
            \min & ~\|W\|_* \\
            s.t. & ~\|\A(W)- b\|_2\leq\epsilon\\
                 & ~W\succeq0, \, W^T = W
\end{array}
\right\}
\end{align}
or its Lagrangian version
\begin{equation}\label{lagrangian}
\min_{W \in \SSn} ~\mu\|W\|_*+\frac{1}{2}\|\A(W)-b\|_2^2,
\end{equation}
where  $\SSn$ is the set of  symmetric positive semidefinite
matrices and $\mu>0$ is a parameter.

In  \cite{ BJT93,GJSW84,Monique97real, Monique00,Monique01-survy},
they studied how to determine whether
 partially specified positive semidefinite matrices can be completed
 to fully specified matrices satisfying certain prescribed
 properties. A number of recent work has also shown that the low-rank solution can be
recovered exactly via minimizing the nuclear norm under certain
conditions \cite{CR08,CT10,RFP08,RXH08}. Several algorithms based on
the interior point method have been proposed in  \cite{ BM03,BM05,
LV09,RFP08,RS05,SRJ05} for solving the semidefinite programming
problem derived from the rank minimization problem
(\ref{min_nuclear1}). Since most of these methods use second-order
information, the memory requirement for computing descent directions
quickly becomes too large as the problem size increases.  Recently,
several fast algorithms using only first-order information have been
developed in \cite{CCS08,GM09,JY09,MGC09,Ma10,TY09}. These
first-order methods, based on function values and gradient
evaluation, cannot yield as high accuracy as interior point methods,
but much larger problems can be solved since no second-order
information needs to be computed and stored.

Motivated by these exciting work, in this paper, we present   two
algorithms for solving the minimum-rank Gram matrix completion
problem (\ref{min_nuclear1}). Our algorithms are based on the
modified fixed point continuation method. By modifying the shrinkage
operator in FPC and using the Barzilai-Borwein technique to compute
explicit dynamically updated step sizes, we get an algorithm, called
modified fixed point continuation method with the  Barzilai-Borwein
technique (MFPC-BB). We  prove the convergence of our algorithm
under certain condition. Some accelerated gradient algorithms were
proposed in
\cite{BT09,JY09,NY83,Nesterov83,Nesterov05,TY09,Tseng08}. These
algorithms  rely on computing the next iterate based not only on the
previous one, but also on two or more previously computed iterates.
These accelerated gradient methods have an attractive convergence
rate of $O(1/k^2)$, where $k$ is the iteration counter. We
incorporate
 this accelerating technique in the  MFPC-BB algorithm to get  
an accelerated fixed point continuation algorithm with the
Barzilai-Borwein technique (AFPC-BB), which  shares the improved
rate $O(1/k^2)$ of the optimal gradient method.

We also notice  that algorithms in  the literature  mostly focus on
recovering a  randomly generated large-scale matrix from incomplete
samples of its entries. Although it has been pointed out briefly  in
\cite{TY09} that these algorithms can be adapted easily to solve the
regularized semidefinite linear least squares problem
(\ref{lagrangian}), it is interesting for us to investigate how to
use these newly developed techniques to compute approximate and
exact rational sum of squares (SOS) decompositions of polynomials
with rational coefficients.

\notation  Let $\Sn\subset\RR^{n\times n}$ denote the space of
symmetric $n\times n$ matrices.  The inner product between two
elements $X, Y\in\Sn$ is denoted by $\langle X,
Y\rangle=\trace(X^TY)$. The Frobenius norm of a matrix $X$ is
denoted by $\|X\|_F$, the nuclear norm by $\|X\|_*$ and the operator
norm (or spectral norm) by $\| X\|_2$.

 The rest of the paper is
organized as follows. In Section 2, we derive the modified fixed
point iterative algorithm for the minimum-rank Gram matrix
completion problem. In Section 3, we establish the convergence
result for the iterations given in Section 2 and prove that it
converges to the optimal solution of the regularized linear least
squares problem (\ref{lagrangian}). In Section 4, we introduce
two techniques to accelerate the convergence of our algorithm 
and present MFPC-BB and AFPC-BB algorithms for solving problem
(\ref{lagrangian}). We demonstrate the performance and effectiveness
of our algorithms through numerical examples for computing
approximate and exact rational sum of squares decompositions of
polynomials with rational coefficients in Section 5.

\section{Modified fixed point iterative algorithm}



Let  $f: \RR^{n_1\times n_2}\rightarrow \RR$  be a convex function,
the subdifferential of $f$ at $X^*\in\RR^{n_1\times n_2}$ denoted by
$\partial f$ is the compact convex set defined by
\begin{align*}
    \partial f(X^*):=\{Z\in\RR^{n_1\times n_2}: f(Y)\geq f(X^*)+\langle Z,Y-X^*\rangle, \forall
~Y\in\RR^{n_1\times n_2}\}.
\end{align*}

Following  discussions in  \cite[Theorem 3.1]{Lewis96} and
\cite{Watson92}, we derive the expression of the subdifferential of
the nuclear norm at a symmetric matrix.

\begin{theorem}\label{sub-nuclearnorm}
Let $W\in\Sn$, then
\begin{gather*}
    \partial\|W\|_*=\{Q^{(1)}Q^{(1)T}-Q^{(2)}Q^{(2)T}+Z: Q^{(i)T} Z=0, i=1,2,
    ~\text{and}~
    \|Z\|_2\leq1
    \},
\end{gather*}
where $Q^{(1)}$ and $Q^{(2)}$ are orthogonal eigenvectors associated
with the positive and negative eigenvalues of $W$ respectively.
\end{theorem}
\begin{proof}
Suppose that the eigenvalues of a symmetric matrix $W$ can be
ordered as
$\lambda_1\geq\cdots\geq\lambda_t>0>\lambda_{t+1}\geq\cdots\geq\lambda_s$,
$\lambda_{s+1}=\cdots=\lambda_n=0$. Let $W=Q\Lambda Q^T$ be a Schur
decomposition of $W$, where $Q\in\RR^{n\times n}$ is an orthogonal
matrix and $\Lambda=\diag(\lambda_1,\ldots,\lambda_n)$. These
matrices can be partitioned as
\begin{gather}\label{part}
    Q = \left(Q^{(1)},Q^{(2)},Q^{(3)}\right),\quad
    \Lambda=\left(
                   \begin{array}{ccc}
                     \Lambda^{(1)} & 0        &0      \\
                     0         & \Lambda^{(2)}&0      \\
                     0         & 0        &\Lambda^{(3)} \\
                   \end{array}
                 \right),
\end{gather}
with $Q^{(1)}, Q^{(2)}, Q^{(3)}$ having $t, s-t, n-s$ columns and
being associated with
$\Lambda^{(1)}=\diag(\lambda_1,\ldots,\lambda_{t})$,
$\Lambda^{(2)}=\diag(\lambda_{t+1},\ldots,\lambda_s)$, and
$\Lambda^{(3)}=\diag(\lambda_{s+1},\ldots,\lambda_n)$, respectively.

Let $\lambda=(\lambda_1,\ldots, \lambda_n)^T$ and recall that
\begin{align*}
    \partial\|\lambda\|_1=\{y\in\RR^n: y_i=1, i=1,\ldots,t; ~y_j=-1, j=t+1,\ldots,s;~|y_k|<1,
    k=s+1,\ldots,n\}.
\end{align*}

Let $Y\in\partial\|W\|_*$, by \cite[Theorem 3.1]{Lewis96}, we have
\begin{align*}
Y=Q\,\diag(d)\,Q^T,
\end{align*}
where $d\in\partial\|\lambda\|_1$. Therefore
\begin{align*}
    Y=Q^{(1)}Q^{(1)T}-Q^{(2)}Q^{(2)T}+Q^{(3)}DQ^{(3)T},
\end{align*}
where $D$ is an $(n-s)\times(n-s)$ diagonal matrix with diagonal
elements  less than $1$ in modulus.

Let $Z=Q^{(3)}DQ^{(3)T}$, we have $Q^{(i)T} Z=0, i=1,2$. Let
$\sigma_1(\cdot)$ denote the largest singular value of a given
matrix, then we have
\begin{align*}
    \|Z\|_2  = Q^{(3)}DQ^{(3)T}\leq  \sigma_1(D)< 1,
\end{align*}
which completes the proof.
\end{proof}

The optimality condition in  \cite[Theorem 2]{MGC09} can be
generalized to the optimality condition for the constrained convex
optimization problem (\ref{lagrangian}).

\begin{theorem}\label{optimal-condition}
Let $f: \Sn \rightarrow \RR$ be a proper convex function, i.e. $f <
+\infty$ for at least one point and $f
> -\infty$ for every point in its domain.
Then $W^*$ is an optimal solution to the problem
\begin{equation}\label{con1}
\min_{W \in \SSn} ~f(W)
\end{equation}
if and only if $W^*\in\SSn$, and there exists a matrix $U\in\partial
f(W^*)$ such that
\begin{equation}\label{con2}
~\langle U, V-W^*\rangle\geq0, ~{\text{for all}}~ V\in\SSn.
\end{equation}
\end{theorem}

\begin{proof}
Suppose $U\in\partial f(W^*)$ and satisfies the inequality condition
(\ref{con2}), hence  
\begin{align*}
    f(V)\geq f(W^*)+\langle U,V-W^*\rangle, \quad\forall~V\in\SSn,
\end{align*}
we have $f(V)\geq f(W^*)$, for all $V\in\SSn$. This shows that $W^*$
is an optimal solution of the problem  (\ref{con1}).

Conversely, suppose $W^*$ is the optimal solution of the problem
(\ref{con1}), and (\ref{con2}) does not hold, i.e., there exists
$U\in\partial f(W^*)$, such that
\begin{align}\label{con3}
    \exists~V\in\SSn,\quad s.t. ~\langle U, V-W^*\rangle<0.
\end{align}
Consider $Z(t)=tW^*+(1-t)V$, where $t\in[0,1]$ is a parameter. Since
$Z(t)$ is on the line segment between $W^*$ and $V$, and $\SSn$ is a
convex set, $Z(t)\in\SSn, \forall ~t\in[0,1]$. By \cite[Theorem
23.4]{Rockafellar72}, the one-sided directional derivative of $f$ at
$Z(1)$ with respect to the vector $W^*-V$ satisfies the following
equation
\begin{align*}
    f^\prime(Z(t);W^*-V)|_{t=1}=f^\prime(W^*;W^*-V)=\sup\{\langle W,W^*-V\rangle: W\in\partial
    f(W^*)\}.
\end{align*}
According to  (\ref{con3}),  we have
\begin{align*}
    f^\prime(Z(t);W^*-V)|_{t=1}\geq \langle U, W^*-V\rangle > 0.
\end{align*}
Therefore,  for a small value $\epsilon>0$, we have
$f(Z(1-\epsilon))< f(W^*)$, which is contradict  to the fact that
$W^*$ is optimal to the problem (\ref{con1}).
\end{proof}

Based on  above theorems, we can introduce a thresholding operator
and extend the fixed point iterative scheme for solving
(\ref{lagrangian}).

\begin{defination}
Suppose $W=Q\Lambda Q^T$ is a  Schur decomposition of  a  matrix $W
\in\Sn$, where $\Lambda=\diag(\lambda_1,\ldots,\lambda_n)$ and $Q$
is a real orthogonal matrix. For any $\nu\geq0$, the matrix
thresholding operator $\D_{\nu}(\cdot)$ is defined as
\begin{align*}
    \D_{\nu}(W):=Q \D_{\nu}(\Lambda) Q^T,\quad
    \D_{\nu}(\Lambda)=\diag(\{\lambda_i-\nu\}_+),
\end{align*}
where $t_+=\max(0,t)$.
\end{defination}


We should point out   that the idea of using the eigenvalue
decomposition of $Y^k$ has also  appeared    in \cite[Remark
3]{TY09}. However, to our best knowledge, there exists no
convergence analysis about the eigenvalue thresholding operator in
the literature.

 Let  $\mu$ and {$\tau$} be positive real numbers and
$X^0$ be an initial starting  matrix. For $k=0,1,2,\cdots$,   we
compute
\begin{align}\label{iter}
\left\{
\begin{array}{ccl}
            Y^k &=& X^k-\tau\A^*(\A(X^k)-b), \\
            X^{k+1} &=& \D_{\tau\mu}(Y^k),
 \end{array}\right.
\end{align}
until a stopping criterion is reached. 

\begin{theorem}\label{eqivalence}
Suppose a matrix $W^*\in \SSn$ satisfies
\begin{enumerate}
  \item $\|\A(W^*)-b\|_2<\mu/n$ for a small positive number $\mu$.
  \item $W^*=\D_{\tau\mu}(h(W^*))$, where $h(\cdot)=I(\cdot)-\tau
\A^*(\A(\cdot)-b)$ and $I(\cdot)$ is an identity operator.
\end{enumerate}
Then $W^*$ is the unique optimal solution of the problem
(\ref{lagrangian}).
\end{theorem}

\begin{proof}
Let $\nu=\tau\mu$ and $Y^*=h(W^*)=W^*+E \in \Sn$, where
$E=-\tau\A^*(\A(W^*)-b).$
We claim that
  $\D_\nu(Y^*)$ is the unique optimal solution to the following problem
  \begin{align}\label{claim-problem}
  \min_{W \in \SSn} ~\nu\|W\|_*+\frac{1}{2}\|W-Y^*\|_F^2,
  \end{align}

In fact, since the objective  function
$\nu\|W\|_*+\frac{1}{2}\|W-Y^*\|_F^2$ is strictly convex, there
exists a unique minimizer, and we only need to prove that it is
equal to $\D_\nu(Y^*)$. Without loss of generality, we assume that
the eigenvalues of $Y^*$ can be ordered as
\begin{align*}
    \lambda_1(Y^*)\geq\cdots\geq\lambda_t(Y^*)\geq\nu>\lambda_{t+1}(Y^*)\geq\cdots>0>\cdots\geq\lambda_{s}(Y^*),
    \lambda_{s+1}(Y^*)=\cdots=\lambda_{n}(Y^*)=0.
\end{align*}
 We compute a Schur decomposition of
$Y^*$ as
\begin{align*}
    Y^*=Q^{(1)}\Lambda^{(1)}Q^{(1)T}+Q^{(2)}\Lambda^{(2)}Q^{(2)T},
\end{align*}
where $\Lambda^{(1)}=\diag(\lambda_1,\ldots,\lambda_t)$,
$\Lambda^{(2)}=\diag(\lambda_{t+1},\ldots,\lambda_s)$,  $Q^{(1)}$
and $Q^{(2)}$ are block matrices corresponding to $\Lambda^{(1)}$
and $\Lambda^{(2)}$ respectively. Let $\widehat{X}=\D_\nu(Y^*)$, we
have
\begin{align*}
    \widehat{X}=Q^{(1)}(\Lambda^{(1)}-\nu I)Q^{(1)T},
\end{align*}
therefore,
\begin{align*}
    Y^*-\widehat{X}=\nu(Q^{(1)}Q^{(1)T}+Z),\quad
    Z=\nu^{-1}Q^{(2)}\Lambda^{(2)}Q^{(2)T}.
\end{align*}
By definition, $Q^{(1)T}Z=0$.
\begin{itemize}
  \item If $\lambda_{t+1}(Y^*)\geq |\lambda_s(Y^*)|$, then
$\|Z\|_2=\lambda_{t+1}(Y^*)/\nu<1$.
  \item Otherwise, let $y=(y_1,\ldots,y_p)^T=\A(W^*)-b\in\RR^p$,
then
\begin{equation*}
    \|E\|_F^2= \tau^2\|\A^* y\|_F^2
          \leq \tau^2 n^2 (y_1^2+\cdots+y_p^2)
             < \tau^2\mu^2.
\end{equation*}
\end{itemize}
Notice that $E\in\Sn$ and $W^*\in\SSn$, by \cite[Theorem
8.1.5]{Golub1996}, we have
\begin{align*}
    \|Z\|_2=\frac{|\lambda_s(Y^*)|}{\nu}=\frac{\max\{|\lambda_1(E)|,|\lambda_n(E)|\}}{\nu}
    \leq\frac{\|E\|_F}{\nu}<1.
\end{align*}
Hence, according to  Theorem \ref{sub-nuclearnorm}, we have
$Y^*-\widehat{X}\in \nu \partial \|\widehat{X}\|_*$, which means
that $0\in \nu \partial \|\widehat{X}\|_*+\widehat{X}-Y^*$. By
Theorem \ref{optimal-condition}, we immediately conclude that
$\D_\nu(Y^*)$ is an optimal solution of the problem
(\ref{claim-problem}).

Since  the objective function of the problem (\ref{lagrangian}) is
strictly convex,  its   optimal solution is also unique. If
$W^*=\D_{\tau\mu}(Y^*)$, by Theorem \ref{optimal-condition}, there
exists a matrix $U\in \nu\partial\|W^*\|_*+W^*-Y^*$ such that
\begin{align*}
    \langle U, V-W^*\rangle\geq0, \quad \forall~V\in\SSn.
\end{align*}
Let $\tildeU=U/\tau$, by substituting $\nu=\tau\mu$ and
$Y^*=W^*-\tau\A^*(\A(W^*)-b)$ into the above subdifferential
function, we have $\tildeU\in \mu\partial\|W^*\|_*+\A^*(\A(W^*)-b)$
satisfying
\begin{align*}
    \langle \tildeU, V-W^*\rangle\geq0,\quad \forall~V\in\SSn.
\end{align*}
By applying Theorem \ref{optimal-condition} once again,  it is true
that $W^*$ is the optimal solution of the problem
(\ref{lagrangian}).
\end{proof}

\section{Convergence analysis}

In this section, we analyze the convergence properties of the
modified fixed point iterative scheme (\ref{iter}). We begin by
recording two lemmas which establish the non-expansivity of the
thresholding operator $\D_\nu(h(\cdot))$.

\begin{lemma}\label{lem1}
 The thresholding operator $\D_\nu$ is non-expansive, i.e., for any
 $X_1, X_2\in\Sn$,
 \begin{align}\label{non-expensive}
    \|\D_\nu(X_1)-\D_\nu(X_2)\|_F\leq\|X_1-X_2\|_F.
 \end{align}
Moreover,
\begin{align*}
    \|X_1-X_2\|_F=\|\D_\nu(X_1)-\D_\nu(X_2)\|_F\Longleftrightarrow
    X_1-X_2=\D_\nu(X_1)-\D_\nu(X_2).
\end{align*}
\end{lemma}
\begin{proof}
Let $X_1=Q^{(1)}\Lambda^{(1)}Q^{(1)T}$ and
$X_2=Q^{(2)}\Lambda^{(2)}Q^{(2)T}$ be Schur decompositions of $X_1$
and $X_2$, respectively, where
\begin{gather*}
    \Lambda^{(1)}=\left(
                    \begin{array}{cc}
                      \diag(\lambda_1) & 0 \\
                              0        & 0 \\
                    \end{array}
                  \right),\quad
    \Lambda^{(2)}=\left(
                    \begin{array}{cc}
                      \diag(\lambda_2) & 0 \\
                              0        & 0 \\
                    \end{array}
                  \right),
\end{gather*}
$\lambda_1=(\alpha_1,\ldots,\alpha_s)^T$ and
$\lambda_2=(\beta_1,\ldots,\beta_t)^T$ are vectors of eigenvalues of
$X_1$ and $X_2$ respectively, and $Q^{(1)}, Q^{(2)}$ are orthogonal
matrices. Suppose that
$\alpha_1\geq\cdots\geq\alpha_k\geq\nu>\alpha_{k+1}\geq\cdots\geq\alpha_s$
and
$\beta_1\geq\cdots\geq\beta_l\geq\nu>\beta_{l+1}\geq\cdots\geq\beta_t$,
then we have
\begin{align*}
    \bX_1:=\D_\nu(X_1)=Q^{(1)}\bL^{(1)}Q^{(1)T},\quad
    \bX_2:=\D_\nu(X_2)=Q^{(2)}\bL^{(2)}Q^{(2)T},
\end{align*}
where
\begin{gather*}
    \widetilde{\Lambda}^{(1)}=\left(
                    \begin{array}{cc}
                      \diag(\bl_1) & 0 \\
                              0        & 0 \\
                    \end{array}
                  \right),\quad
    \widetilde{\Lambda}^{(2)}=\left(
                    \begin{array}{cc}
                      \diag(\bl_2) & 0 \\
                              0        & 0 \\
                    \end{array}
                  \right),
\end{gather*}
$\bl_1=(\alpha_1-\nu,\ldots,\alpha_k-\nu)$ and
$\bl_2=(\beta_1-\nu,\ldots,\beta_l-\nu)$. Therefore, we have
\begin{align*}
    &\|X_1-X_2\|_F^2-\|\bX_1-\bX_2\|_F^2 \\
    &=\Tr((X_1-X_2)^T(X_1-X_2))-\Tr((\bX_1-\bX_2)^T(\bX_1-\bX_2))\\
    &=\Tr(X_1^TX_1-\bX_1^T\bX_1+X_2^TX_2-\bX_2^T\bX_2)-2\Tr(X_1^TX_2-\bX_1^T\bX_2)\\
    &=\sum_{i=1}^s\alpha_i^2-\sum_{i=1}^k(\alpha_i-\nu)^2+\sum_{i=1}^t\beta^2-\sum_{i=1}^l(\beta_i-\nu)^2-2\Tr(X_1^TX_2-\bX_1^T\bX_2).
\end{align*}

It is known that
for symmetric matrices $X,Y$,
\begin{gather*}
    \Tr (X\,Y)\leq \lambda(X)^T\lambda(Y),
\end{gather*}
with equality if and only if there exists an orthogonal matrix $Q$
such that
\begin{gather*}
    X=Q\diag(\lambda(X))Q^T,\quad Y=Q\diag(\lambda(Y))Q^T,
\end{gather*}
where $\lambda(X),\lambda(Y)$ are the vectors of eigenvalues of $X$
and $Y$ respectively (see \cite[Theorem 2.2]{Lewis96}). Hence,
without loss of generality, assuming $k\leq l\leq s\leq t$, we have
\begin{align*}
    \Tr(X_1^TX_2-\bX_1^T\bX_2)&=\Tr((X_1-\bX_1)^T(X_2-\bX_2)+(X_1-\bX_1)^T\bX_2+\bX_1^T(X_2-\bX_2))\\
                              &\leq \lambda(X_1- \bX_1)^T\lambda(X_2-
                              \bX_2)+ \lambda(X_1- \bX_1)^T\lambda(
                              \bX_2)+\lambda(\bX_1)^T\lambda(X_2-
                              \bX_2)\\
                              & \leq
                              \sum_{i=1}^l\alpha_i\nu+\sum_{i=l+1}^s\alpha_i\beta_i+
                              \sum_{i=1}^k(\beta_i-\nu)\nu+\sum_{i=k+1}^l\alpha_i(\beta_i-\nu)
\end{align*}
Therefore,
\begin{align*}
    \|X_1-X_2\|_F^2-\|\bX_1-\bX_2\|_F^2
     &\geq
    \sum_{i=1}^s\alpha_i^2-\sum_{i=1}^k(\alpha_i-\nu)^2+\sum_{i=1}^t\beta^2-\sum_{i=1}^l(\beta_i-\nu)^2\\
    &\hspace{0.5cm} -2(\sum_{i=1}^l\alpha_i\nu+\sum_{i=l+1}^s\alpha_i\beta_i+\sum_{i=1}^k(\beta_i-\nu)\nu+\sum_{i=k+1}^l\alpha_i(\beta_i-\nu))\\
    & \geq (\sum_{i=l+1}^s\alpha_i^2+\sum_{i=l+1}^t\beta_i^2-2\sum_{i=l+1}^s\alpha_i\beta_i)+\sum_{i=k+1}^l(2\beta_i\nu-\nu^2+\alpha_i^2-2\alpha_i\beta_i).
\end{align*}
Since $t\geq s$ and $\alpha_i^2+\beta_i^2-2\alpha_i\beta_i\geq0$, we
obtain
\begin{gather*}
    \sum_{i=l+1}^s\alpha_i^2+\sum_{i=l+1}^t\beta_i^2-2\sum_{i=l+1}^s\alpha_i\beta_i\geq0.
\end{gather*}
Moreover, since the function
$g(x)=2\beta_ix-x^2+\alpha_i^2-2\alpha_i\beta_i$ is monotonically
increasing in $[-\infty, \beta_i]$, and
$\alpha_i\leq\nu\leq\beta_i$, $i=k+1,\ldots,l$,
\begin{gather*}
    2\nu\beta_i-\nu^2+\alpha_i^2-2\alpha_i\beta_i>0, \quad
    i=k+1,\ldots,l.
\end{gather*}
Hence, we have
\begin{gather*}
    \|X_1-X_2\|_F^2-\|\bX_1-\bX_2\|_F^2\geq0,
\end{gather*}
i.e., (\ref{non-expensive}) holds.

Furthermore, if $\|X_1-X_2\|_F=\|\D_\nu(X_1)-\D_\nu(X_1)\|_F$, then
$s=t, k=l$ and $\alpha_i=\beta_i, i=k+1,\ldots,s$, which further
implies that $\Lambda^{(1)}-\bL^{(1)}=\Lambda^{(2)}-\bL^{(2)}$ and
$\Tr((X_1-\bX_1)^T(X_2-\bX_2))$ achieves its maximum. Hence,  there
exists an orthogonal matrix $Q$ such that
\begin{gather*}
    X_1-\bX_1=Q(\Lambda^{(1)}-\bL^{(1)})Q^T=Q(\Lambda^{(2)}-\bL^{(2)})Q^T=X_2-\bX_2,
\end{gather*}
which implies that
\begin{gather}\label{eq2}
    X_1-X_2=\D_\nu(X_1)-\D_\nu(X_2).
\end{gather}
Suppose (\ref{eq2}) holds, then
$\|X_1-X_2\|_F=\|\D_\nu(X_1)-\D_\nu(X_2)\|_F$, which completes the
proof.
\end{proof}

The following lemma and its proof are analogous to results in
\cite{{HYZ08},MGC09}.

\begin{lemma}\label{lem2}
Suppose that the step size $\tau$ satisfies $\tau\in(0,
2/\|\A\|_2^2)$. Then the operator $h(\cdot)=I(\cdot)-\tau
\A^*(\A(\cdot)-b)$ is non-expansive, i.e., for any
 $X_1, X_2\in\Sn$,
\begin{gather*}
    \|h(X_1)-h(X_2)\|_F\leq\|X_1-X_2\|_F.
\end{gather*}
Moreover, we have
\begin{gather*}
    \|h(X_1)-h(X_2)\|_F=\|X_1-X_2\|_F\Longleftrightarrow
    h(X_1)-h(X_2)=X_1-X_2,
\end{gather*}
where $I(\cdot)$ is an identity operator.
\end{lemma}

We now claim that the modified fixed point iterations (\ref{iter})
converge to the optimal solution of the  problem (\ref{lagrangian}).

\begin{theorem}\label{convergence}
Let $\tau\in(0, 2/\|\A\|_2^2)$ and $W^*\in \SSn$ satisfy
\begin{enumerate}
  \item $\|\A(W^*)-b\|_2<\mu/n$ for a small positive number $\mu$.
  \item $W^*=\D_{\tau\mu}(h(W^*))$, where $h(\cdot)=I(\cdot)-\tau
\A^*(\A(\cdot)-b)$.
\end{enumerate}
Then the sequence $\{X^k\}$ obtained via modified fixed point
iterations (\ref{iter}) converges to $W^*$.
\end{theorem}
\begin{proof}
Let $\nu=\tau\mu$. Since both $\D_\nu(\cdot)$ and $h(\cdot)$ are
non-expansive, $\D_\nu(h(\cdot))$ is also non-expansive. Therefore,
$\{X^k\}$ lies in a compact set and must have a limit point. Suppose
 $\bX=\lim_{j\longrightarrow\infty}X^{k_j}$ satisfying  $\|\A(\bX)-b\|_2<\mu/n$.
By $W^*=\D_\nu(h(W^*))$, we have
\begin{align*}
    \|X^{k+1}-W^*\|_F=\|\D_\nu(h(X^k))-\D_\nu(h(W^*))\|_F\leq\|h(X^k)-h(W^*)\|_F \leq\|X^k-W^*\|_F,
\end{align*}
which means that the sequence $\{\|X^k-W^*\|_F\}$ is monotonically
non-increasing. Therefore
\begin{gather*}
    \lim_{k\longrightarrow\infty}\|X^k-W^*\|_F=\|\bX-W^*\|_F,
\end{gather*}
where $\bX$ can be any limit point of $\{X^k\}$. By the continuity
of $\D_\nu(h(\cdot))$, we have
\begin{gather*}
    \D_\nu(h(\bX))=\lim_{j\longrightarrow\infty}\D_\nu(h(X^{k_j}))=\lim_{j\longrightarrow\infty}X^{k_j+1},
\end{gather*}
i.e., $\D_\nu(h(\bX))$ is also a limit point of $\{X^k\}$.
Therefore, we have
\begin{gather*}
    \|\D_\nu(h(\bX))-\D_\nu(h(W^*))\|_F=\|\D_\nu(h(\bX))-W^*\|_F=\|\bX-W^*\|_F.
\end{gather*}
Using Lemma \ref{lem1} and Lemma \ref{lem2} we obtain
\begin{gather*}
    \D_\nu(h(\bX))-\D_\nu(h(W^*))=h(\bX)-h(W^*)=\bX-W^*,
\end{gather*}
which implies $ \D_\nu(h(\bX))=\bX$. By Theorem \ref{eqivalence},
$\bX$ is the optimal solution to the problem (\ref{lagrangian}),
i.e., $\bX=W^*$. Hence, we have
\begin{gather*}
    \lim_{k\longrightarrow\infty}\|X^k-W^*\|_F=0,
\end{gather*}
i.e., $\{X^k\}$ converges to its unique limit point $W^*$.
\end{proof}

\section{Implementation}

This section provides implementation details of the modified FPC
algorithm for  solving the minimum-rank Gram matrix completion
problem.


\subsection{Evaluation of the eigenvalue thresholding  operator}

The main computational cost of the modified FPC algorithm is
computing the Schur decompositions. 
Following the strategies in \cite{CCS08,TY09}, we
 use PROPACK \cite{PROPACK} in Matlab to compute a partial Schur decomposition of a symmetric matrix.

%

PROPACK can not automatically compute only eigenvalues greater than
a given threshold $\nu$. To use this package, we must predetermine
the number $s_{k}$ of eigenvalues of $Y^k$ to compute at the $k$-th
iteration.   Suppose $X^{k}=Q^{k-1} \Lambda^{k-1}(Q^{k-1})^T$, we
set $s_k$ equal to the number of diagonal entries of $\Lambda^{k-1}$
that are no less than $\varepsilon_k \|\Lambda^{k-1}\|_2$, where
$\varepsilon_k$ is a
small positive number.  
Notice  that $s_k$ is non-increasing.    If $s_k$ is too small,  the
non-expansive property (\ref{non-expensive}) of the thresholding
operator $\D_\nu$ may be violated.  We increase  $s_k$  by 1 if the
non-expansive property is violated 10 times \cite{MGC09}.

\subsection{Barzilai-Borwein technique}

In \cite{MGC09}, the authors always set the parameter $\tau=1$ since
their operator $\A$ is generated by randomly sampling a subset of
$p$ entries from matrices with i.i.d. standard Gaussian entries. For
this linear map, the Lipschitz constant for the objective function
of (\ref{lagrangian}) is 1. According to Theorem \ref{convergence},
convergence for the Gram matrix completion problem is guaranteed
provided that $\tau\in(0,~2/\|\A\|_2^2)$. This choice is, however,
too conservative and the convergence is typically slow.

There are many  ways to select a step size.  For simplicity, we
describe a strategy, which is  based on the Barzilai-Borwein method
\cite{BB88}, for choosing the step size $\tau_k$. Let
$g(\cdot)=\A^*(\A(\cdot)-b)$ and $g^k=\A^*(\A(X^k)-b)$. We perform
the shrinkage iteration (\ref{iter}) along the negative gradient
direction $g^k$ of the
smooth function $\frac{1}{2}\|\A(X^k)-b\|_2^2$,  
 then apply the thresholding operator $\D_\nu(\cdot)$ to
accommodate the non-smooth term $\|X\|_*$. Hence, it is natural to
choose $\tau_k$ based on the function $\frac{1}{2}\|\A(X^k)-b\|_2$
alone. Let
\begin{gather*}
    \Delta X=X^k-X^{k-1},\quad \Delta g=g^k-g^{k-1}.
\end{gather*}
The Barzilai-Borwein step provides a two-point approximation to the
secant equation underlying quasi-Newton method, specifically,
\begin{align*}
    \tau_k=\frac{\langle \Delta X, \Delta g\rangle}{\langle\Delta g, \Delta
    g\rangle},\quad\mbox{or}\quad \tau_k=\frac{\langle \Delta X, \Delta X\rangle}{\langle\Delta X, \Delta
    g\rangle}.
\end{align*}
In order to avoiding  the parameter $\tau_k$ being either too small
or too large, we take
\begin{align*}
     \tau_k=\max\{\tau_{min}, \min\{\tau_k, \tau_{max}\}\},
\end{align*}
where $0<\tau_{min}<\tau_{max}<\infty$ are fixed parameters.

The idea of using the BB step to accelerate the convergence of
gradient algorithms has also appeared in \cite{WYGZ09}.

\subsection{Algorithms}

As suggested in \cite{HYZ08,MGC09,TY09}, we adopt a continuation
strategy to solve the regularized linear least squares problem
(\ref{lagrangian}).
For the problem (\ref{lagrangian}) with a target parameter
$\bar{\mu}$  being  a moderately small number, we propose solving a
sequence of problems (\ref{lagrangian}) defined by an decreasing
sequence $\mu_k$. When a new problem, associated with $\mu_{k+1}$,
is to be solved, the approximate solution for the current problem
with $\mu_k$ is used as the starting point. We use the parameter
$\eta$ to determine the rate of reduction of the consecutive
$\mu_k$, i.e.,
\begin{gather*}
    \mu_{k+1}=\max(\eta\mu_k,\bar{\mu}),\quad k=1,\ldots,L-1.
\end{gather*}

Our modified fixed point continuation iterative scheme with the
Barzilai-Borwein technique for solving (\ref{lagrangian}) is
outlined below.
\begin{ouralgorithm}{MFPC-BB} \label{alg:MFPC-BB}

 \Inspec  Parameters $0<\tau_{min}<\tau_0<\tau_{max}<\infty$,
 $\mu_1>\bar{\mu}>0$, $\eta>0$ and a tolerance $\epsilon>0$

 \Outspec A numeric Gram matrix.

 \begin{description}
  \item[-] Set $X^0=0$.
  \item[-] \textbf{For} $\mu=\mu_1,\ldots,\mu_L$, \textbf{do}

                   \begin{enumerate}
                     \item  Choose a step size $\tau_k$ via the BB technique such that $\tau_{min}\leq\tau_k\leq\tau_{max}$.
                     \item  Compute $Y^k=X^k-\tau_k\A^*(\A(X^k)-b)$ and a Schur decomposition of $Y^k=Q^k~\Lambda^k~(Q^k)^T$.
                     \item  Compute $X^{k+1}=Q^k~\D_{\tau_k\mu_k}(\Lambda^k)~(Q^k)^T$.
                   \end{enumerate}

  \item[-] \textbf{If} the stop criterion is true, \textbf{then return} $X_{\opt}$.
  \item[-] \textbf{end for}.
 \end{description}
\end{ouralgorithm}

However, as shown in \cite{BT09, JY09, TY09}, the above algorithm
may converge as $O(1/k)$.
 Very recently, alternative algorithms that
could speed up the performance of the gradient method FPC have been
proposed in \cite{JY09,TY09}. These algorithms rely on computing the
next iterate based not only on the previous one, but also  on two or
more previously computed iterates. We incorporate this new
accelerating technique in our MFPC-BB algorithm to solve the affine
constrained low-rank Gram matrix completion problem
(\ref{lagrangian}). The accelerated algorithm, called AFPC-BB, keeps
the simplicity of MFPC-BB but shares the improved rate $O(1/k^2)$ of
the optimal gradient method.

\begin{ouralgorithm}{AFPC-BB} \label{alg:AFPC-BB}

 \Inspec  Parameters $0<\tau_{min}<\tau_0<\tau_{max}<\infty$,
 $\mu_1>\bar{\mu}>0$, $\eta>0$ and tolerance $\epsilon>0$

 \Outspec A numeric Gram matrix.

 \begin{description}
  \item[-] Set $X^0=0$.
  \item[-] \textbf{For} $\mu=\mu_1,\ldots,\mu_L$, \textbf{do}

                   \begin{enumerate}
                     \item  Choose a step size $\tau_k$ via the BB technique such that $\tau_{min}\leq\tau_k\leq\tau_{max}$.
                     \item  Compute $Z^k=X^k+\frac{t_{k-1}-1}{t_k}(X^k-X^{k-1})$.
                     \item  Compute $Y^k=Z^k-\tau_k\A^*(\A(Z^k)-b)$ and a Schur decomposition of $Y^k=Q^k~\Lambda^k~(Q^k)^T$.
                     \item  Compute $X^{k+1}=Q^k~\D_{\tau_k\mu_k}(\Lambda^k)~(Q^k)^T$.
                     \item  Compute $t_{k+1}=\frac{1+\sqrt{1+4 t_k^2}}{2}$.
                   \end{enumerate}

  \item[-] \textbf{If} the stop criterion is true, \textbf{then return} $X_{\opt}$.
  \item[-] \textbf{end for}.
 \end{description}
\end{ouralgorithm}

The following theorem shows that by performing the gradient step at
the matrix $Z^k$ instead of at the approximate solution $X^k$, the
convergence rate of the MFPC-BB method can be accelerated to
$O(1/k^2)$.
\begin{theorem}\cite{JY09,TY09}
Let \{$X^k$\} be the sequence generated by the  AFPC-BB algorithm.
Then for any $k>1$, we have
\begin{gather}\label{accel}
    F(X^k)-F(X^*)\leq\frac{C\|X^*-X^0\|_F^2}{(k+1)^2}, 
\end{gather}
where $C$ is a constant, $F(X)$ is the objective function and $X^*$
is the optimal solution  of the problem (\ref{lagrangian}).
\end{theorem}

\section{Numerical experiments}

In this section, we report the performance of our modified FPC
algorithms for writing a real positive semidefinite polynomial as a
sum of minimum number of squares of polynomials.
In our tests, we generate positive semidefinite matrices
$W\in\QQ^{n\times n}$ with rank $r$ by randomly sampling an $n\times
r$ factor $L$ with rational entries and setting $W=LL^{\text{T}}$.
After multiplying  the matrix $W$  by a monomial vector $\mon$ and
its transpose, we obtain a positive semidefinite polynomial
\begin{gather*}
    f(x)=\mon^T \cdot W \cdot \mon\in\QQ[x].
\end{gather*}
Replacing  entries in $W$ by  parameters, expanding the right-hand
side of the equality and matching coefficients of the monomials, we
obtain a set of linear equations  which can be written as
\begin{equation}\label{lincons}
{\A}(W) = b,
\end{equation}
where  $\A$ is the linear map from  $\Sn$ to $\RR^p$.

Since the SOS representation of a nonnegative polynomial is in
general not unique, the solution $X_{\opt}$ returned by MFPC-BB and
AFPC-BB algorithms probably doesn't correspond to the constructed
rational Gram matrix $W$. Therefore, in stead of setting relative
error equal to $\|X_{\opt}-W\|_F/\|W\|_F$, which is used in
\cite{CCS08,CR08,MGC09,TY09}, we choose to measure the accuracy of
the computed solution $X_{\opt}$ by the relative error defined by:
\begin{equation}\label{rel-err}
    \err := \frac{\|\A(X_{\opt})-b\|_2}{\|b\|_2}.
\end{equation}
The  relative error also gives us a stopping criterion  for the
MFPC, MFPC-BB, AFPC-BB algorithms in our numerical experiments. 
 We declared
that the Gram matrix is approximately recovered if the relative
error is less than  a given tolerance denoted by $\epsilon$. 

An $n\times n$ symmetric matrix of rank $r$ depends on $d_r =
n(2n-r+1)/2$ degrees of freedom.  Let  $FR=d_r/p$  be the ratio
between the degrees of freedom in an $n\times n$ symmetric matrix of
rank $r$ and the number of linear constrains  defined in
(\ref{lincons}). If $FR$ is large (close to  1), recovering $W$
becomes harder as the number of measurements is close to the degrees
of freedom. Conversely, if $FR$ is close to zero, recovering $W$
becomes easier.
Note that if $FR>1$,  there might have  an infinite number of
matrices with rank $r$ satisfying given affine constraints.

Throughout the experiments, we choose an initial matrix  $X^0$ to be
a zero matrix. For each test, we make an initial estimate of the
value $L=\|\A\|_2^2$ which is the smallest Lipschitz constant of the
gradient of $\frac{1}{2}\|\A X-b\|_2^2$. We
 set the Barzilai-Borwein parameters $\tau_{max}=10/L$ and
$\tau_{min}=10^{-3}/L$. The thresholds $10$ and $10^{-3}$ are found
after some experimentations.

We have implemented the MFPC-BB and AFPC-BB algorithms in MATLAB,
using PROPACK package to evaluate partial eigenvalue decompositions.
All runs are conducted on  a HP xw8600  workstation with an Inter
Xeon(R) 2.67GHz CPU and 3.00 GB of RAM.

\subsection{Numerical experiments on random  Gram matrix
completion problems}

In the first series of test, we set $\epsilon=5\times10^{-3}$ and
compare the performance of the MFPC, MFPC-BB and AFPC-BB algorithms
without continuation technique to solve (\ref{lagrangian}) for
randomly generated matrix completion
problems with moderate dimensions. 
In order to   see  the convergence behaviors of MFPC, MFPC-BB and
AFPC-BB clearly, we compute the full Schur decompositions at each
iteration.  

Table \ref{tab1} reports the degree of freedom ratio $FR$, the
number of
iterations, 
and the error (\ref{rel-err}) of the three algorithms MFPC, MFPC-BB,
AFPC-BB. As can be seen from Table \ref{tab1}, on the condition that
these three algorithms achieve  similar errors, MFPC-BB provides
better performance with less number of iterations than MFPC,  while
AFPC-BB outperforms the other two algorithms greatly in terms of the
number
of iterations.  

\begin{table*}[!htb]
\begin{center}
\begin{tabular}{cccc|cc|cc|cc}
\hline   \multicolumn{4}{c|}{Problems}& \multicolumn{2}{c|}{MFPC}  & \multicolumn{2}{c|}{MFPC-BB} & \multicolumn{2}{c}{AFPC-BB} \\
\hline   n  &  r  &  p  & $FR$  &$\#$~iter     &  error  &$\#$~iter     & error   &$\#$~iter     & error  \\
\hline
\hline  100 & 10  & 579 & 1.6494&  140    & 4.99e-3 & 75      & 4.95e-3 & 31      & 4.76e-3\\
\hline  200 & 10  & 1221& 1.6011&  187    & 4.99e-3 & 105     & 4.97e-3 & 37      & 4.88e-3\\
\hline  500 & 10  & 5124& 0.9670&  632    & 4.99e-3 & 499     & 4.99e-3 & 66      & 4.90e-3\\
\hline
\end{tabular}
\end{center}
\caption{\label{tab1} Comparison of MFPC, MFPC-BB and AFPC-BB,
without using continuation technique.}
\end{table*}

In Figure 1 and Figure 2,  
we plot the relative error $\|\A(X^k)-b\|_2/\|b\|_2$ and
approximation error $\|X^k-X_{opt}\|_F$ versus the iteration number
of these three methods on recovering a randomly generated $500
\times 500$ Gram matrix with  rank $10$ respectively. We terminate
these three algorithms when the relative error (\ref{rel-err}) is
below $5\times10^{-3}$. We observe  that in both cases AFPC-BB
converges much fast than MFPC-BB and MFPC. The comparison of MFPC-BB
and MFPC clearly shows that the Barzilai-Borwein technique is quite
effective in accelerating the convergence of the MFPC algorithm.

\begin{figure}[!htbp]
\hspace*{-2mm}
\begin{minipage}[t]{0.4\linewidth}
    \centering
    \includegraphics[width=3.5in]{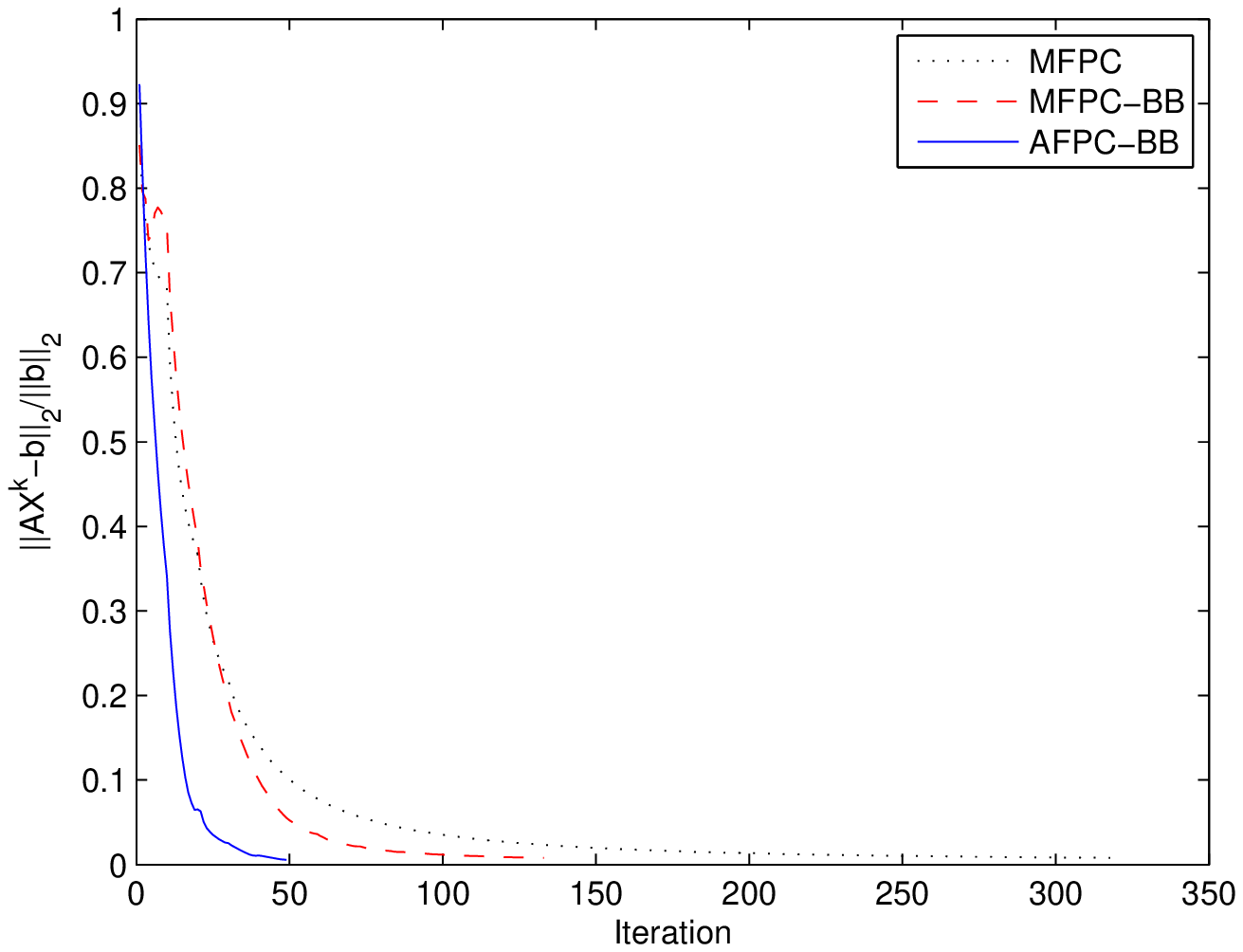}
    \caption{ Relative error versus iteration number for a
problem with $n=500, r=10$.}
    \label{Fig 1a}
\end{minipage}
\hspace*{15mm}
\begin{minipage}[t]{0.4\linewidth}
    \centering
    \includegraphics[width=3.5in]{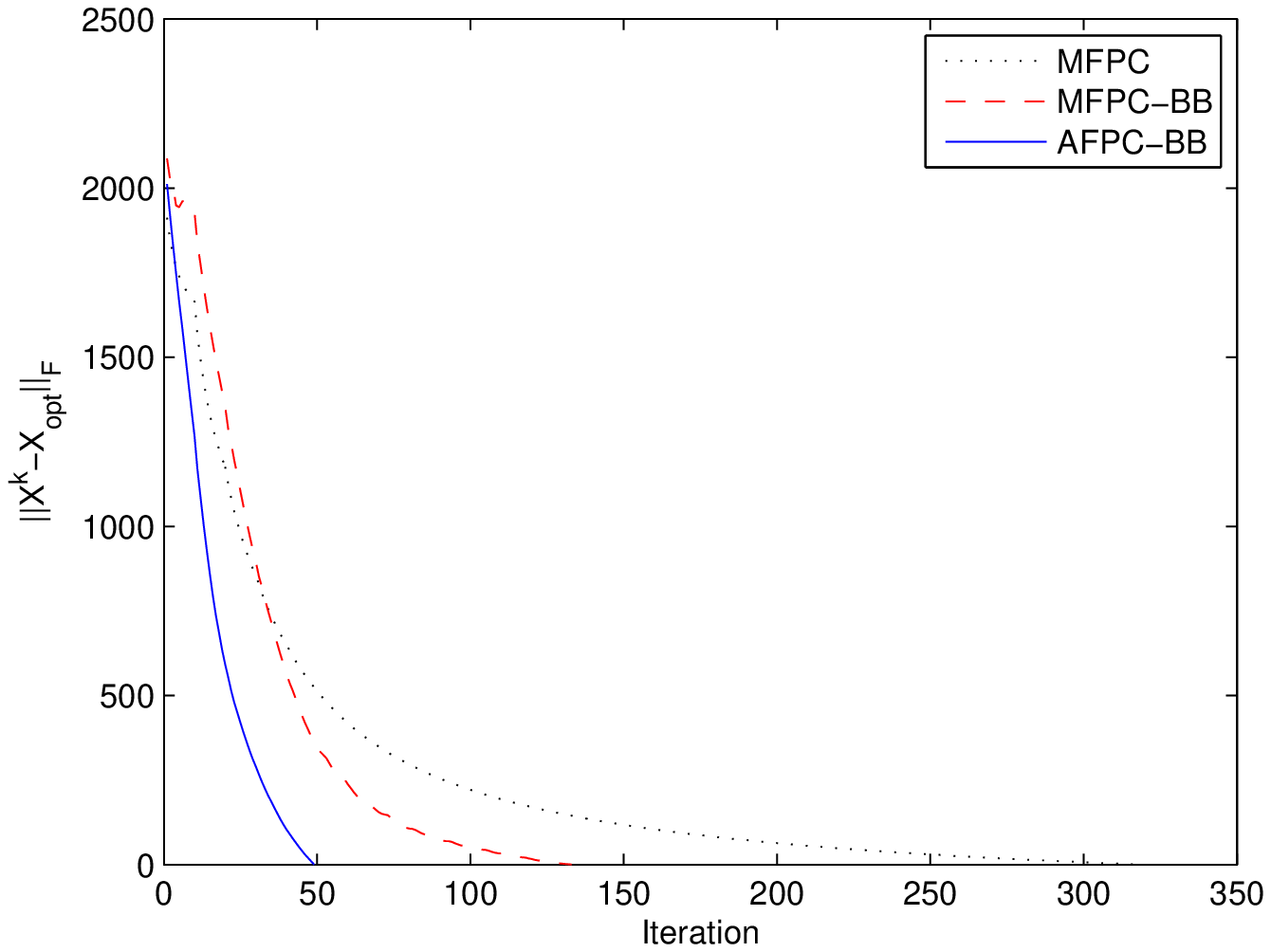}
     \caption{ Approximation error versus iteration number for a
problem with $n=500, r=10$.}
    \label{Fig 1b}
\end{minipage}

\end{figure}

In  Table \ref{tab2}, we report the performance of the AFPC-BB
algorithm with continuation technique on randomly generated Gram
matrix completion problems. We use PROPACK to compute partial
eigenvalues and eigenvectors. We set the regularization parameter in
problem (\ref{lagrangian}) to be $\bar{\mu}=10^{-4}\|\A^*b\|$ and
$\mu_1=1/4\|\A^*b\|$. The update strategy for $\mu_k$ is
$\max(1/4\mu_{k-1},\bar{\mu})$ whenever the stopping criterion is
satisfied with $\epsilon= 10^{-3}$.

\begin{table*}[!htb]
\begin{center}
\begin{tabular}{cccc|ccc}
\hline   \multicolumn{4}{c|}{Problems}& \multicolumn{3}{c}{Results}  \\
\hline   n  &  r  &  p  &  FR  &$\#$~iter&     time  &  error  \\
\hline
\hline  100 & 10  & 579 &1.6494&  76     & 1.48e+0  & 9.64e-4 \\
\hline  500 & 10  & 3309&1.4974&  80     & 2.35e+1  & 9.90e-4 \\
\hline  1000& 10  &10621&0.9372&  165    & 1.41e+2  & 9.95e-4 \\
\hline  1000& 50  &10621&4.5923&  120    & 1.10e+2  & 9.89e-4 \\
\hline  1500& 10  &25573&0.5848&  271    & 6.04e+2  & 9.96e-4 \\
\hline  1500& 50  &25573&2.8849&  156    & 4.59e+2  & 9.83e-4 \\
\hline
\end{tabular}
\end{center}
\caption{\label{tab2} Numerical results for AFPC-BB on random Gram
matrix completion problems.}
\end{table*}

As indicated in the table, it takes  the AFPC-BB algorithm  fewer
than 300 iterations on the average  and  less than 15 minutes to
solve all  problems in our experiments.   In addition, for most of
these problems, $FR$ is larger than $1$.  Especially, $FR$ is up to
$4.5923$ for the problem with $n=1000, r=50$.   To our best
knowledge, nobody has considered solving matrix completion problems
in this situation yet.  It is rather surprising that the original
random Gram matrix with low rank  can be recovered given only such a
small number of affine constraints.

\subsection{Exact rational sum of squares certificates}

The numerical Gram matrix $W$ returned by the AFPC-BB algorithm
satisfies
\begin{equation}\label{appSOS}
  f(x)\approx \mon^T\cdot W\cdot \mon,\quad W \succeq 0.
\end{equation}
In order to derive an exact SOS decomposition of $f$,  we need to
start with an approximate Gram matrix with high accuracy. Although
first-order methods are often the only practical option for
large-scale problems, it has also been observed   that the sequence
$\{X^k\}$ computed by the AFPC-BB algorithm converges quite slowly
to an optimal solution $W^*$.  Therefore, we apply the
structure-preserving Gauss-Newton  iterations (see
\cite{KLYZ08,{KLYZ09}}) to refine the Gram matrix $W$ with low rank
returned by the AFPC-BB algorithm:  we choose a rank $r$ which is
less than or equal to the rank of $W$ and compute the truncated
L${}^{\text{T}}$DL decomposition of $W$ to
 obtain an  approximate SOS decomposition 
\begin{equation*}
f(x) \approx \sum_{i=1}^{r} (\sum_{\alpha} c_{i,\alpha}x
^{\alpha})^2,
\end{equation*}
then  apply standard Gauss-Newton  iteration to compute $\Delta
c_{i,\alpha}x^{\alpha}$ such that
\begin{equation}\label{GN}
f(x)= \sum_{i=1}^{r} (\sum_{\alpha } c_{i,\alpha}x^{\alpha} +\Delta
c_{i,\alpha}x^{\alpha})^2+ O(\sum_{i=1}^r (\sum_\alpha \Delta
c_{i,\alpha} x^\alpha)^2).
\end{equation}
 The matrix $W$ is updated  accordingly to $W+\Delta W$ and the
iteration is stopped when the backward error
\begin{gather}\label{poly-err}
   \theta=\|f(x)-\mon^T \cdot W \cdot \mon\|_2
\end{gather}
is less than the given tolerance $\epsilon$.   If $\theta$ remains
greater than $\epsilon$ after several Gauss-Newton  iterations, we
may increase the precision or use different $r$ and try Gauss-Newton
iterations again. After converting the refined matrix $W$ into a
rational matrix, we use the orthogonal projection technique in
\cite{KLYZ08,{KLYZ09}} to construct an exact rational SOS
decomposition for the nonnegative polynomial $f$.



It is interesting to notice that the AFPC-BB algorithm provides a
low-rank Gram matrix to seed  Gauss-Newton iterations while most of
the SDP solvers GloptiPoly~\cite{gloptipoly},
SOSTOOLS~\cite{sostools}, YALMIP \cite{YALMIP},
SeDuMi~\cite{Sturm99}, SDPT3 \cite{Toh98sdpt3} and
SparsePOP~\cite{SparsePOP} usually return a Gram matrix with
 maximum rank (see \cite[Theorem 2.1]{KRT97}). For example,
  we consider a randomly generated Gram matrix completion
problem with $n=200$, $r=5$, which is created in  the same way
described at the beginning of Section 5. The smallest
 $10$ singular values of the numerical Gram matrix computed by  SeDuMi   are
 \[ 3.527, 2.779, 2.445, 1.369, 1.184, 0.964,
0.627, 0.101, 0.485,
   0.161, 0.069.
\]
However, the rank of the numerical Gram matrix returned by the
AFPC-BB algorithm is $14$. We notice that by applying Gauss-Newton
iterations to the low-rank Gram matrix computed by AFPC-BB, it is
usually much easy to recover an exact SOS decomposition of the
nonnegative polynomial.

In \cite{Ma10}, we have used the MFPC-BB algorithm  to  successfully
 recover the exact sums of squares of nonnegative polynomials  in
\cite{KLYZ09}. 
%
%

In the following two tables, we compare the performance of the
AFPC-BB algorithm and the SDP solver SeDuMi  for recovering low rank
Gram matrices from affine constraints on the same randomly generated
examples.  We also show the effectiveness of Gauss-Newton iterations
run in Maple with $Digits=14$  in refining the numerical Gram
matrix. These tables report the number of affine constraints $p$,
the degree of freedom ratio $FR$, the backward error $\theta$, the
rank of the Gram matrix and the running time in seconds. Table
\ref{tab:upper2} also shows the smallest singular value $\sigma_n$
of numerical Gram matrices returned by SeDuMi. We set
$\epsilon=5\times10^{-4}$ in the AFPC-BB algorithm, which is small
enough to guarantee very good recoverability.

\begin{table*}[!htb]
\begin{center}
\begin{tabular}{cccc|ccc|ccc}
\hline   \multicolumn{4}{c|}{Examples} & \multicolumn{3}{c|}{AFPC-BB} &\multicolumn{3}{c}{Gauss-Newton iteration} \\
\hline  n  & r  & p    &  FR  & rank  &$\theta $& time   & rank  &$\theta $& time \\
\hline
\hline  50 & 5  & 255  &0.9412&   9   &  6.874  & 4.06e-1  &  5    & 1.443e-5 & 4.08e+0  \\
\hline  100& 5  & 579  &0.8463&   9   &  0.860  & 1.75e+0  &  5    & 1.935e-9 & 2.98e+1 \\
\hline  150& 5  & 896  &0.8259&   13  &  2.758  & 7.09e+0  &  5    & 4.023e-8 & 6.28e+1 \\
\hline  200& 5  & 1221 &0.8108&   14  &  3.629  & 1.07e+1  &  5    & 4.030e-5 & 4.69e+2 \\
\hline  300& 5  & 1932 &0.7712&   14  &  22.315 & 2.32e+1  &  5    & 1.379e-9 & 5.61e+2   \\
\hline  400& 5  & 2610 &0.7624&   15  &  12.515 & 6.23e+1  &  5    & 5.825e-5 & 1.22e+3   \\
\hline  500& 5  & 5124 &0.4859&   17  &  24.829 & 5.33e+1  &  5    & 1.479e-5 & 7.92e+3   \\
\hline
\end{tabular}
\end{center}
\caption{\label{tab:upper1} Exact SOS certificates via AFPC-BB and
Gauss-Newton iterations.}
\end{table*}

\begin{table*}[!htb]
\begin{center}
\begin{tabular}{cccc|cc|ccc}
\hline   \multicolumn{4}{c|}{Examples} & \multicolumn{2}{c|}{SDP} &\multicolumn{3}{c}{Gauss-Newton iteration} \\
\hline  n  & r  & p    &  FR   &$\sigma_n $  & time   & rank  &$\theta $& time \\
\hline
\hline  50 & 5  & 255  & 0.9412&  0.701 & 1.03e+0 &  6    & 3.769e-8&  1.59e+1  \\
\hline  100& 5  & 579  & 0.8463&  0.042 & 7.77e+0 &  7    & 2.438e-10& 6.88e+1 \\
\hline  150& 5  & 896  & 0.8259&  0.069 & 1.24e+1 &  7    & 1.883e-10& 2.23e+2 \\
\hline  200& 5  & 1221 & 0.8108&  0.069 & 6.58e+1 &  7    & 4.666e-9 & 8.21e+2 \\
\hline  300& 5  & 1932 & 0.7712&  0.442 & 2.84e+2 &  7    & 5.679e-10& 1.30e+3 \\
\hline  400& 5  & 2610 & 0.7712&  0.114 & 3.94e+2 &  8    & 9.249e-10& 5.00e+3  \\
\hline  500& 5  & 5124 & 0.4859&  0.001 & 2.14e+3 &  ---  &   ---    & ---  \\
\hline
\end{tabular}
\end{center}
\caption{\label{tab:upper2} Approximate SOS certificates via SDP and
Gauss-Newton iterations.}
\end{table*}

As indicated in Table \ref{tab:upper1}, using the AFPC-BB algorithm,
 we can compute numerical low-rank Gram
matrices  very efficiently. Moreover,  for each example, we can use
Gauss-Newton iterations (\ref{GN}) to refine the Gram matrix
returned by AFPC-BB to relatively high accuracy, e.g. $10^{-5}$.
By rounding every entry of the refined matrix to the nearest
integer, we can easily  recover a rational Gram matrix with rank $5$
which gives the exact SOS representation of the nonnegative
polynomial.

As indicated in Table \ref{tab:upper2}, for the same  examples,
numerical Gram matrices  returned by SeDuMi have full rank  for the
given tolerance $10^{-3}$,  while the rank of matrices returned by
AFPC-BB are relatively small. We seed the numerical Gram matrices
returned by SeDuMi to Gauss-Newton iterations,  the ranks of the
refined matrices are always larger than $5$ in order to guarantee
the convergence of the Gauss-Newton iterations. Furthermore, we are
not yet able to recover exact SOS decompositions even though
backward errors $\theta$ have been reduced to the order of
$10^{-10}$.

From Table \ref{tab5}, it is also interesting to notice, if we
decrease the degree of freedom ratio $FR$ by choosing a sparse
monomial vector $m_d(x)$, it is possible to recover the exact SOS
representation of the nonnegative polynomial from  the numerical
low-rank Gram matrix returned by the AFPC-BB algorithm, without
running Gauss-Newton iterations.

\begin{table*}[!htb]
\begin{center}
\begin{tabular}{cccc|ccc|c}
\hline   \multicolumn{4}{c|}{Problems}& \multicolumn{3}{c|}{AFPC-BB} & \multicolumn{1}{c}{Rational SOS} \\
\hline   n  &  r  &  p  &  FR  &$\#$~iter&     time  &  error       &  time\\
\hline
\hline  50  &  5  & 608 &0.3947&  45     & 4.38e-1  & 5.84e-4       &  1.09e-1\\
\hline  100 &  5  & 1167&0.4199&  100    & 1.97e+0  & 8.72e-4       &  3.59e-1\\
\hline  150 &  5  & 1703&0.4345&  217    & 5.91e+0  & 9.96e-4       &  8.12e-1\\
\hline  200 &  5  & 2249&0.4402&  239    & 9.11e+0  & 9.99e-4       &  1.50e+0\\
\hline  300 &  5  & 3544&0.4204&  327    & 2.11e+1  & 9.77e-4       &  3.45e+0\\
\hline  400 & 10  &10078&0.3924&  151    & 2.46e+1  & 9.52e-4       &  1.14e+1\\
\hline  500 & 20  &24240&0.4047&  142    & 4.48e+1  & 4.70e-4       &  4.65e+1\\
\hline  1000& 10  &27101&0.3673&  436    & 3.70e+2  & 4.97e-4       &  1.38e+2\\
\hline  1000& 50  &95367&0.5114&  395    & 6.56e+2  & 9.99e-5       &  1.41e+3\\
\hline  1500& 10  &45599&0.3280&  554    & 1.00e+3  & 4.99e-4       &  3.10e+2\\
\hline
\end{tabular}
\end{center}
\caption{\label{tab5}\small  Exact SOS certificates via AFPC-BB.}
\end{table*}

%
%
%
%



\def\refname{\Large\bfseries References}
\bibliographystyle{plain}
\bibliography{yma}

\end{document}